\documentclass[12pt]{amsart}
\usepackage{amssymb,latexsym}
\usepackage{amsthm}
\usepackage{involutory_functions}
\begin{document}
\title[On the power series of involutory functions]{On the power series of involutory functions}
\author{Alfred Schreiber}
\address{Department of Mathematics\\
         University of Flensburg\\
				 Auf dem Campus 1\\
				 24943 Flensburg, Germany}
\email{info@alfred-schreiber.de}
\sloppy
\begin{abstract}
It is shown that the coefficients of any involutory function $f$ represented as a power series can be expressed in terms of multivariable Lah polynomials. This result is based on the fact that any such $f~(\neq\text{identity})$ can be regarded as a (compositional) conjugate of negative identity. Moreover, a constructive proof of this statement is given.
\end{abstract}

\fussy
\maketitle
\section{Preliminaries}
In the following, a \emph{function} shall be understood as a formal power series with coefficients from a fixed commutative field $\const$ of characteristic zero. For our purposes, we mostly restrict ourselves to (compositionally) \emph{invertible} functions, i.\,e., $f\in\const[[x]]$ with $f\comp\inv{f}=\inv{f}\comp f=\id$, where $\inv{f}$ denotes the (unique) inverse of $f$ and $\id$ denotes the identity: $\id(x)=x$. According to a simple criterion \cite[Proposition\,5.4.1]{stan2001}, $f=f(x)=\sum_{n\geq0}f_n\frac{x^n}{n!}$ is invertible if and only if $f_0=f(0)=0$ and $f_1=f'(0)\neq0$. The invertible functions form a non-commutative group $\invfuncs$ with respect to the composition $\comp$. \emph{Involutory functions} (or \emph{involutions}) are the elements $f$ in this group with $f\comp f=\id$. The two simplest examples of involutory functions are the identity $\id$ (\emph{trivial} involution) and the negative identity $-\id$.

A sequence of constants $g_1,g_2,g_3,\ldots\in\const$, $g_1\neq0$, uniquely determines a function $g=g(x)=g_{1}x+g_{2}x^2/2!+g_{3}x^3/3! +\cdots\in\invfuncs$, from which we can in turn recover the $g_1,g_2,g_3,\ldots$ by differentiation as Taylor coefficients: $D^n(g)(0)=g_n$ ($n=1,2,3,\ldots$). The operator $D$ here denotes the unique extension of the ordinary algebraic deri\-vation on $\const[x]$ (with $D(\id)=1$) to the algebra $\const[[x]]$ of formal power series (cf. \cite{bend1980}). $D^n$ denotes, as usual, the $n$-th iterate of $D$.

Of special interest is the task of finding the general Taylor coefficient of a composite function $f\comp g$ ($g\in\invfuncs$, $f$ not necessarily invertible). Its well-known solution is given by the famous formula of Fa\`{a} di Bruno (see \cite[p.\,137]{comt1974}, \cite[eqs. (1.3), (1.4)]{schr2015}), which in modern notation is
\begin{equation}\label{FdB_formula}
	D^n(f\comp g)(0)=\sum_{k=0}^{n}f_{k}B_{n,k}(g_1,\ldots,g_{n-k+1}).
\end{equation}
As is common, the symbol $B_{n,k}$ stands for the \emph{partial Bell} (or \emph{exponential}) \emph{polynomials} belonging to $\const[X_1,\ldots,X_{n-k+1}]$ (see, e.g., \cite[p.\,133]{comt1974}, \cite[p.\,65]{masc2016}). Replacing the function $g$ with its inverse $\inv{g}$ leads to the following nontrivial counterpart of \eqref{FdB_formula}:
\begin{equation}\label{FdB_formula_dual}
	D^n(f\comp\inv{g})(0)=\sum_{k=0}^{n}f_{k}A_{n,k}(g_1,\ldots,g_{n-k+1}).
\end{equation}
Here the family $A_{n,k}$ consists of (Laurent) polynomials $\in\const[X_1^{-1},X_2,\ldots,X_{n-k+1}]$ which are `ortho-inverse' companions of the $B_{n,k}$ in the sense that for $1\leq k\leq n$
\begin{equation}\label{ortho_inversion}
	\sum_{j=k}^{n}A_{n,j}B_{j,k}=\kronecker{n}{k}
\end{equation}
with $\kronecker{n}{n}=1$, $\kronecker{n}{k}=0$ for $n\neq k$ (Kronecker's symbol).

\begin{rem}
The fundamental properties (explicit representations, recurrences, inverse relations, reciprocity laws) of these double-indexed polynomial families are treated in detail in \cite{schr2015,schr2021a,schr2021b}. The collective term \emph{multivariate Stirling polynomials} (MSP) \emph{of the first and second kind} was proposed for $A_{n,k}$ and $B_{n,k}$ because the associated coefficient sums $A_{n,k}(1,\ldots,1)$ and $B_{n,k}(1,\ldots,1)$ turn out to be just the signed Stirling numbers of the first kind and the Stirling numbers of the second kind, respectively. This happens when $g(x)=\text{e}^x-1$ is chosen.
\end{rem}

Specializing $f=\id^k/k!$, we obtain from \eqref{FdB_formula} and \eqref{FdB_formula_dual} after a short calculation
\begin{align*}
	D^n(\tfrac{g^k}{k!})(0)=D^n(\tfrac{\id^k}{k!}\comp g)(0)=B_{n,k}(g_1,\ldots,g_{n-k+1}),\\
	D^n(\tfrac{\inv{g}^k}{k!})(0)=D^n(\tfrac{\id^k}{k!}\comp\inv{g})(0)=A_{n,k}(g_1,\ldots,g_{n-k+1}),
\end{align*}
which clearly illustrates the meaning of the MSP. --- Finally, note the following remarkable identity \cite[Theorem\,5.3]{schr2015}, which holds for $1\leq k\leq n$:
\begin{equation}\label{dualization_identity}
	A_{n,k}=B_{n,k}(A_{1,1},\ldots,A_{n-k+1,1}).
\end{equation}
The reader will find a set of tables for the MSP in \cite[p.\,2471]{schr2015} and \cite[p.\,51--54]{schr2021b}. The partial Bell polynomials $B_{n,k}$, $1\leq k\leq n\leq 12$, are tabulated in \cite[p.\,307--308]{comt1974}.

\section{The coefficients of involutory functions}
With the aid of the Fa\`{a} di Bruno formula \eqref{FdB_formula} it is possible to characterize the general Taylor coefficient $f_n$ of an involutory function $f$ by a recurrence.
\begin{thm}\label{involution_coeffs}
Let $f\in\invfuncs$ be any function with $f\neq\id$ and $f_n=D^n(f)(0)$. Then, $f$ is involutory if and only if there is a sequence of constants $a_1,a_2,a_3,\ldots\in\const$ such that for every $n\geq1$
\begin{equation}\label{recurrence_thm1}
	f_n=
	\begin{cases}
		-1, &\text{if $n=1$;}\\
		a_{n/2}, &\text{if $n$ even;}\\
		\frac{1}{2}\sum\limits_{k=2}^{n-1}f_{k}B_{n,k}(-1,f_2,\ldots,f_{n-k+1}), &\text{if $n\geq3$ odd.}
	\end{cases}
\end{equation}
\end{thm}
\begin{proof}
We calculate the coefficients of the two sides of $f\comp f=\id$ separately. This yields $D^n(\id)(0)=\kronecker{n}{1}$ and further by applying \eqref{FdB_formula} to $D^n(f\comp f)(0)$ the following statement, which is equivalent to $f$ being involutory:
\begin{equation}\label{basic_eq1}
	\sum_{k=1}^{n}f_{k}B_{n,k}(f_1,\ldots,f_{n-k+1})=\kronecker{n}{1}\quad\text{for all~}n\geq1.
\end{equation}
For $n=1$ equation \eqref{basic_eq1} becomes $1=\kronecker{1}{1}=f_{1}B_{1,1}(f_1)=f_1^2$. Thus, to prove $f_1=-1$, we must show that $f_1=1$ leads to a contradiction. For this purpose, we first consider \eqref{basic_eq1} for $n\geq2$ and by splitting off the first and the last summand on the left side of the equation we obtain
\begin{equation}\label{basic_eq2}
	f_{1}B_{n,1}(f_1,\ldots,f_n)+\sum_{k=2}^{n-1}f_{k}B_{n,k}(f_1,\ldots,f_{n-k+1})=-f_{n}B_{n,n}(f_1). 
\end{equation}
We now show by induction that indirectly assuming $f_1=1$ implies $f_n=0$ for all $n\geq2$, which contradicts $f\neq\id$. In the case $n=2$ we have $f_2=f_{1}B_{2,1}(f_1,f_2)=-f_{2}B_{2,2}(f_1)=-f_2$, and therefore $f_2=0$. Now assume $f_j=0$ for $3\leq j\leq n$ (induction hypothesis). From this follows by \eqref{basic_eq2}
\[
	f_{n+1}=f_{1}B_{n+1,1}(f_1,\ldots,f_{n+1})+0=-f_{n+1}B_{n+1,n+1}(f_1)=-f_{n+1}, 
\]
hence $f_{n+1}=0$. Thus $f_1=-1$ is proved.

If we now substitute this value into \eqref{basic_eq2}, we obtain after a few transformations the following equivalent statement: 
\begin{equation}\label{basic_eq3}
	\sum_{k=2}^{n-1}f_{k}B_{n,k}(-1,f_2,\ldots,f_{n-k+1})=(1+(-1)^{n+1})f_n\quad\text{for all~}n\geq2.
\end{equation}
For odd $n\geq3$ the right-hand side becomes $2f_n$ and the third line of the assertion \eqref{recurrence_thm1} is immediately clear. In the case of an even $n\geq2$, the right-hand side vanishes, with the consequence that the same coefficients, namely $f_2,\ldots,f_{n-1}$, occur both in the equation \eqref{basic_eq3} for $n$ and in that for its odd predecessor $n-1$, while $f_n$ itself remains undetermined. The coefficients $f_2,f_4,f_6,\ldots$ can thus be taken as a sequence of arbitrary constants $a_k:=f_{2k}$, $k\geq1$, on which the solution system of the equation \eqref{basic_eq1} depends. This shows the rest of the assertion.
\end{proof}

\begin{exm}
We illustrate here the statement of Theorem \ref{involution_coeffs} by listing the first odd indexed coefficients of any nontrivial involution $f$ in the form of polynomial expressions of the arbitrary constants $f_{2k}=a_k$ ($1\leq k\leq4$):
\begin{align*}
	f_1&=-1,\\
	f_3&=-\tfrac{3}{2}a_1^2,\\
	f_5&=15a_1^4-\tfrac{15}{2}a_{1}a_{2},\\
	f_7&=-\tfrac{4095}{4}a_1^6+\tfrac{945}{2}a_1^{3}a_2-\tfrac{35}{2}a_2^2-14a_{1}a_{3},\\
	f_9&=\tfrac{411075}{2}a_1^8-\tfrac{208845}{2}a_1^5a_2+7875a_1^2a_2^2+2205a_1^3a_3-105a_{2}a_3-\tfrac{45}{2}a_{1}a_4.
\end{align*}
It remains open (if not questionable) whether a closed-form representation for the general coefficient $f_{2k-1}$ can be achieved from the calculation procedure \eqref{recurrence_thm1}.
\end{exm}

\section{Multivariable Lah polynomials}
The recursive representation of the coefficients established in the previous section is based upon a direct evaluation of the defining functional equation $f\comp f=\id$ using the Fa\`{a} di Bruno formula. As shown in \cite{schr2021a,schr2021b}, it is nevertheless possible to obtain a closed representation for the $f_n$ by using a decomposition of $f$ in the form $f(x)=g(-\inv{g}(x))$ with some $g\in\invfuncs$; in argument-free notation: 
\begin{equation}\label{conjugate_notation}
	f=g\comp(-\id)\comp\inv{g}.
\end{equation}
Throughout Subsection 5.5 of \cite{schr2021a}, which among other things deals with the re\-presentation of involutory functions, use is made of \eqref{conjugate_notation} without proof. Instead, there is a reference to McCarthy \cite{mcca1980}, which, however, misses the point, since this paper is about continuous bijective involutions $\reals\to\reals$ and based on some set-theoretic arguments. On closer inspection it becomes clear that the statements derived in this context cannot be simply carried over to functions in the form of formal power series. Thus, the possibility of a conjugate representation \eqref{conjugate_notation} requires a proof of its own, which also closes the gap in \cite[Subsection\,5.5]{schr2021a}. 

In preparing this, we first show that equation (3.1) allows us to give the general form of the coefficients of an involution.
\begin{thm}\label{lah_representation}
Let $f,g\in\invfuncs$ be arbitrarily given with $f\neq\id$, $f_n=D^n(f)(0)$ and $g_n=D^n(g)(0)$. Then, $f=g\comp(-\id)\comp\inv{g}$ holds if and only if $f_n=L_{n,1}(g_1,g_2,\ldots,g_{2[n/2]})$ for all $n\geq1$.
\end{thm}
Here the polynomials $L_{n,1}$ form a subfamily of the double-indexed \emph{multivariable Lah polynomials} introduced in \cite{schr2021a} and defined by
\begin{equation}\label{lah_polys}
	L_{n,k}=\sum_{j=k}^n(-1)^{j}A_{n,j}B_{j,k}.
\end{equation}
\begin{rem}\label{lah_numbers}
Similar to the Stirling polynomials mentioned above, the naming is justified by the hardly surprising fact that the associated number sequence $L_{n,k}(1,\ldots,1)$, more precisely the sum of coefficients of $L_{n,k}$, is just that of the well-known signed Lah numbers $(-1)^n\frac{n!}{k!}\binom{n-1}{k-1}$; see \cite[p.\,19]{schr2021a}.
\end{rem}
\begin{exm}\label{lah_subfamily}
The following list contains the first 6 members of the subfamily $L_{n,1}$:
\footnotesize{
{\allowdisplaybreaks
\begin{align*}
 L_{1,1}&=-1,\quad L_{2,1}=\frac{2 X_2}{X_1^2},\quad L_{3,1}=-\frac{6 X_2^2}{X_1^4},\\
 L_{4,1}&=\frac{30 X_2^3}{X_1^6}-\frac{8 X_3 X_2}{X_1^5}+\frac{2 X_4}{X_1^4}, \\
 L_{5,1}&=-\frac{210 X_2^4}{X_1^8}+\frac{120 X_3 X_2^2}{X_1^7}-\frac{30 X_4
   X_2}{X_1^6}. \\
 L_{6,1}&=\frac{1890 X_2^5}{X_1^{10}}-\frac{1680 X_3 X_2^3}{X_1^9}+\frac{420 X_4 X_2^2}{X_1^8}+\frac{140 X_3^2 X_2}{X_1^8}-\frac{12 X_5 X_2}{X_1^7}+\frac{2 X_6}{X_1^6}-\frac{40 X_3 X_4}{X_1^7}
\end{align*}
} 
}
\end{exm}
\normalsize
We now prove Theorem \ref{lah_representation}.

\begin{proof}
The equation $f=g\comp(-\id)\comp\inv{g}$ being satisfied is equivalent to the statement that $f_n=D^n(g\comp(-\id)\comp\inv{g})(0)$ holds for all $n\geq1$. Hence, only $D^n(g\comp((-\id)\comp\inv{g}))(0)$ needs to be evaluated. We first determine the coefficients of $(-\id)\comp\inv{g}$. By means of \eqref{FdB_formula_dual} we obtain
\begin{equation}\label{coeffs_post}
	D^r((-\id)\comp\inv{g})(0)=\sum_{j=1}^{r}D^j(-\id)(0)A_{r,j}(g_1,\ldots,g_{r-j+1})=-A_{r,1}(g_1,\ldots,g_r).
\end{equation}
Next, we apply the Fa\`{a} di Bruno formula \eqref{FdB_formula} as well as equation \eqref{dualization_identity}, observing that the $B_{n,k}$ are homogeneous polynomials of degree $k$:
\begin{align*}
	f_n&=D^n(g\comp((-\id)\comp\inv{g}))(0)\\
	   &=\sum_{k=1}^{n}g_{k}B_{n,k}(-A_{1,1}(g_1),-A_{2,1}(g_1,g_2),\ldots,-A_{n-k+1,1}(g_1,\ldots,g_{n-k+1}))\\
		 &=\sum_{k=1}^{n}(-1)^{k}A_{n,k}(g_1,\ldots,g_{n-k+1})B_{k,1}(g_1,\ldots,g_k)\\
		 &=L_{n,1}(g_1,g_2,\ldots,g_{2[n/2]}).\qedhere
\end{align*}
\end{proof}
In order to prove that any involution admits a decomposition of the kind \eqref{conjugate_notation}, we need the following two fundamental properties of Lah polynomials.
\begin{prop}\label{lah_selfinverse}\[
	\sum_{j=k}^{n}L_{n,j}L_{j,k}=\kronecker{n}{k}\qquad(1\leq k\leq n).	\]
\end{prop}
\begin{proof}
A direct consequence from \eqref{lah_polys} and \eqref{ortho_inversion}. See \cite[Proposition 5.6]{schr2021a}.
\end{proof}
The next statement says that the Lah polynomials $L_{n,k}$ can be represented using the partial Bell polynomials with the indeterminates $X_j$ replaced by the subfamily terms $L_{j,1}$ for each $j\geq1$.
\begin{prop}\label{bell_representability}\[
	L_{n,k}=B_{n,k}(L_{1,1},\ldots,L_{n-k+1,1})\qquad(1\leq k\leq n). \]
\end{prop}
\begin{rem}
This result is derived in \cite{schr2021a} directly from the context of the proof of Theorem 5.3 (see Corollary 5.1), so it is not immediately obvious whether or not it makes use of the decomposition \eqref{conjugate_notation}, the possibility of which remains to be proved here. In anyway, the proof that follows is independent of \eqref{conjugate_notation}.
\end{rem}
\begin{proof}
Assume $g$ to be any function $\in\invfuncs$. Putting $\widetilde{g}:=g\comp(-\id)$ and $h:=\widetilde{g}\comp\inv{g}$ we get by Theorem \ref{lah_representation} 
\begin{equation}\label{bell_rep_eq1}
	h_n=D^n(h)(0)=L_{n,1}(g_1,g_2,\ldots).
\end{equation}
From \eqref{lah_polys} and \eqref{dualization_identity} follows
\begin{align}\label{bell_rep_eq2}
	L_{n,k}&=\sum_{j=k}^{n}B_{n,j}(A_{1,1},A_{2,1},\ldots)B_{j,k}(-X_1,X_2,-X_3,\ldots)\notag\\
	\intertext{and after replacing $X_r$ with $g_r$ ($r=1,2,\ldots$)}
	L_{n,k}(g_1,g_2,\ldots)&=\sum_{j=k}^{n}B_{n,j}(\inv{g}_1,\inv{g}_2,\ldots)B_{j,k}(\widetilde{g}_1,\widetilde{g}_2,\ldots),
\end{align}
where $\inv{g}_r=D^r(\inv{g})(0)$ and $\widetilde{g}_r=D^r(\widetilde{g})(0)$ for $r=1,2,\ldots$. If we now take a closer look at the right-hand side of \eqref{bell_rep_eq2}, we see that it is equal to $B_{n,k}(h_1,h_2,\ldots)$, according to a formula established by Jabotinsky \cite{jabo1947,jabo1953} (cf. also \cite[Section\,3.7]{comt1974} and the \emph{Second Composition Rule} in \cite[Theorem~4.2\,(i)]{schr2021a}). This yields
\[
	L_{n,k}(g_1,g_2,\ldots)=B_{n,k}(L_{1,1}(g_1),L_{2,1}(g_1,g_2),\ldots)
\]
and the asserted identity if we recall that polynomials over an infinite integral domain (here the field $\const$) that give rise to the same polynomial function are identical.
\end{proof}
Combining Proposition \ref{lah_selfinverse} and \ref{bell_representability} yields the following lemma, which we need to prove our main results in the next section.
\begin{lem}\label{aux-claim}
For all $n\geq2$ we have
\[
	\sum_{k=2}^{n-1}L_{k,1}B_{n,k}(L_{1,1},\ldots,L_{n-k+1,1})=(1+(-1)^{n+1})L_{n,1}.
\]
\end{lem}
\begin{proof}
From Proposition \ref{lah_selfinverse} and \ref{bell_representability} we conclude
\[
	\sum_{k=1}^{n}L_{k,1}B_{n,k}(L_{1,1},\ldots,L_{n-k+1,1})=\kronecker{n}{1}.
\]
Taking into account $L_{1,1}B_{n,1}(L_{1,1},\ldots,L_{n,1})=-L_{n,1}$, this results in
\[
	\sum_{k=2}^{n}L_{k,1}B_{n,k}(L_{1,1},\ldots,L_{n-k+1,1})=L_{n,1}+\kronecker{n}{1}
\]
and further
\[
	\sum_{k=2}^{n-1}L_{k,1}B_{n,k}(L_{1,1},\ldots,L_{n-k+1,1})=L_{n,1}+\kronecker{n}{1}-L_{n,1}B_{n,n}(L_{1,1}).
\]
Here, the right-hand side is equal to $L_{n,1}-L_{n,1}\cdot(-1)^n$ for $n\geq2$, which proves the assertion.
\end{proof}

\section{The main results}
We are now in a position to prove that any nontrivial involution can be represented in terms of the equation \eqref{conjugate_notation}.
\begin{thm}\label{main_result_1}
Let $f$ be any function from $\invfuncs$ with $f\neq\id$. Then, $f$ is involutory if and only if there exists $g\in\invfuncs$ such that $f=g\comp(-\id)\comp\inv{g}$.
\end{thm}
\begin{proof}
The sufficiency (`if') is immediate, since $g\comp(-\id)\comp\inv{g}$ is clearly involutory. For necessity (`only if') assume $f$ to be any involutory function $\neq\id$. Then, by Theorem \ref{involution_coeffs} there are constants $a_1,a_2,\ldots\in\const$ such that the coefficients $f_n=D^n(f)(0)$ satisfy the recurrence \eqref{recurrence_thm1}. We now construct the function we are looking for
\[
	g(x)=\sum_{n\geq1}g_n\frac{x^n}{n!}
\]
in the following way: The odd indexed coefficients $g_1,g_3,g_5,\ldots$ are chosen as arbitrary constants $\in\const$, where $g_1\neq0$; for even $n\geq2$, we define $g_n$ recursively by setting
\begin{equation}\label{thm3_eq1}
	g_n=\frac{1}{2}\sum_{k=2}^{n}f_{k}B_{n,k}(g_1,\ldots,g_{n-k+1}).
\end{equation}
We must prove that $g$ satisfies \eqref{conjugate_notation}. According to Theorem \ref{lah_representation}, it suffices to show that $L_{n,1}(g_1,g_2,\ldots)=f_n$ holds for all $n\geq1$. The following induction is split into two parts, depending on the parity of $n$. 

The cases $n=1$ and $n=2$ are settled readily by observing Theorem~\ref{involution_coeffs} and Example~\ref{lah_subfamily} as follows: 
\begin{align*}
	L_{1,1}(g_1,g_2,\ldots)&=-1=f_1,\\
	L_{2,1}(g_1,g_2,\ldots)&=\frac{2g_2}{g_{1}^2}\underset{\eqref{thm3_eq1}}{=}\frac{1}{g_{1}^2}\cdot f_{2}B_{2,2}(g_1)=f_2,
\end{align*}
where $f_2=a_1$ (arbitrary constant). 

Induction hypothesis: Assume $L_{r,1}(g_1,g_2,\ldots)=f_r$ for $3\leq r\leq n-1$.

First, we consider the case of an odd $n\geq3$. From Lemma \ref{aux-claim} we get
\begin{alignat*}{2}
	L_{n,1}&=\frac{1}{2}\sum_{k=2}^{n-1}L_{k,1}B_{n,k}(L_{1,1},L_{2,1},\ldots)\\
\intertext{and by replacing each $X_j$ with $g_j$}
L_{n,1}(g_1,g_2,\ldots)&=\frac{1}{2}\sum_{k=2}^{n-1}L_{k,1}(g_1,g_2,\ldots)B_{n,k}(L_{1,1}(g_1,&g_2,\ldots),L_{2,1}(g_1,g_2,\ldots),\ldots)\\
	      &=\frac{1}{2}\sum_{k=2}^{n-1}f_{k}B_{n,k}(f_1,f_2,\ldots,f_{n-k+1})&\text{(induction hypothesis)}\\
	      &=f_n.&\text{(Theorem \ref{involution_coeffs}, \eqref{recurrence_thm1})}
\end{alignat*}
Let us now assume that $n\geq4$ is even. Recalling $B_{j,1}=X_j$ equation \eqref{lah_polys} can be written as
\[
	L_{n,1}=\sum_{j=1}^{n}(-1)^{j}X_{j}A_{n,j}.
\]
We apply to this the \emph{inversion of sequences} based on \eqref{ortho_inversion} and explained in \cite[Corollary 5.2]{schr2015} (cf. also \cite[Proposition 5.4]{schr2021a}) thus obtaining
\[
	\sum_{k=1}^{n}L_{k,1}B_{n,k}=(-1)^{n}X_n
\]
and after substituting $g_j$ for $X_j$:
\[
	\sum_{k=1}^{n}L_{k,1}(g_1,g_2,\ldots)B_{n,k}(g_1,g_2,\ldots)=(-1)^{n}g_n.
\]
From this follows by the induction hypothesis 
\[
	\sum_{k=1}^{n-1}f_{k}B_{n,k}(g_1,g_2,\ldots)=(-1)^{n}g_n-g_{1}^{n}L_{n,1}(g_1,g_2,\ldots).
\]
Now consider the first summand on the left-hand side, which is $f_{1}B_{n,1}(g_1,g_2,\ldots)=-g_n$. We split it off and then solve the equation for $L_{n,1}(g_1,g_2,\ldots)$:
\begin{align*}
	L_{n,1}(g_1,g_2,\ldots)&=-\frac{1}{g_{1}^n}\sum_{k=2}^{n-1}f_{k}B_{n,k}(g_1,g_2,\ldots)+\frac{g_n}{g_{1}^n}(1+(-1)^n).\\
\intertext{By \eqref{thm3_eq1} we have}	
  2g_n&=\sum_{k=2}^{n-1}f_{k}B_{n,k}(g_1,g_2,\ldots)+f_{n}B_{n,n}(g_1)
\intertext{and thus from the above}	
 L_{n,1}(g_1,g_2,\ldots)&=-\frac{1}{g_{1}^n}(2g_n-f_{n}g_{1}^n)+\frac{2g_n}{g_{1}^n}=f_n.\qedhere
\end{align*}
\end{proof}
From the Theorems \ref{lah_representation} and \ref{main_result_1} follows immediately our second main result, which describes the form of the coefficients of involutory functions.
\begin{thm}\label{main_result_2}
Let $f$ be any function from $\invfuncs$, $f\neq\id$. Then, $f$ is involutory if and only if there exist constants $g_1 (\neq0),g_2,g_3,\ldots\in\const$ such that the general Taylor coefficient $f_n$ of $f$ meets $f_n=L_{n,1}(g_1,g_2,\ldots,g_{2[n/2]})$.
\end{thm}

\section{Discussion}
Let us look more closely at the statement of Theorem \ref{main_result_1}. If there exists $g\in\invfuncs$ satisfying $f=g\comp(-\id)\comp\inv{g}$, this says that $f$ and the negative identity $-\id$ are conjugate elements in the group $(\invfuncs,\comp)$. Conversely, the nontrivial part of the statement says that every involution $f\neq\id$ of $\invfuncs$ is indeed a conjugate of $-\id$. In summary, \emph{all elements of} $\invfuncs$ \emph{of order 2 form just one conjugacy class}.

One can by no means expect the function $g$ to be uniquely determined by an equation $f=g\comp(-\id)\comp\inv{g}$. As a quite simple example, consider $g(x)=\text{e}^x-1$ and the corresponding sequence $g_1=g_2=\ldots=1$, for which $f_n=L_{n,1}(1,\ldots,1)=(-1)^{n}n!$ (cf. Remark \ref{lah_numbers}) and the resulting involutory function becomes
\begin{equation}\label{simple_example}
		f(x)=-x+x^2-x^3+x^4-\cdots=-\frac{x}{1+x}.
\end{equation}
There are also other sequences leading to the same involution, for instance the coefficients $g_j=c^j$, ($j=1,2,3,\ldots$) with $c\in\const$, $c\neq0$, belonging to the function $g(x)=\text{e}^{cx}-1$ and also satisfying $L_{n,1}(c,c^2,c^3,\ldots)=(-1)^{n}n!$. So, the obvious question is how to describe the set of all functions $g\in\invfuncs$ satisfying $f=g\comp(-\id)\comp\inv{g}$ for a given involution $f$. The answer is easily obtained by assuming any two functions $g,h\in\invfuncs$ for which 
\begin{equation}\label{two_generators}
	g\comp(-\id)\comp\inv{g}=h\comp(-\id)\comp\inv{h}.
\end{equation}
If we set $\psi:=\inv{g}\comp h$, \eqref{two_generators} can be equivalently transformed into $\psi\comp(-\id)=(-\id)\comp\psi$, which shows $\psi$ is an odd function. Conversely, the general solution $h$ of the equation $f=h\comp(-\id)\comp\inv{h}$ is obtained in the form $h=g\comp\psi$, where $g$ is any particular solution and $\psi$ is some odd function from $\invfuncs$. Note that the subgroup of odd functions is the centralizer of $-\id$ in $\invfuncs$.

Applying this to the involutory function $f$ in \eqref{simple_example}, it results that the equation $f(x)=g(-\inv{g}(x))$ holds if and only if $g$ has the form $g(x)=\text{e}^{\psi(x)}-1$ with $\psi$ being an odd function. Suitable examples are $\psi(x)=cx$, which gives the above sequence $c,c^2,c^3,\ldots$, or $\psi(x)=\sin x$. Of course, the seemingly unwieldy sequence $1,1,0,-3,-8,-3,56,217,64,-2951,\ldots$ produced by $\text{e}^{\sin x}-1$ nevertheless leads to the same coefficients $L_{n,1}(1,1,0,-3,-8,-3,\ldots)=(-1)^{n}n!$, as the reader may verify from the list in Example~\ref{lah_subfamily} for some values.
%

\nocite* 
\end{document}